\documentclass[12pt]{elsarticle}
\makeatletter
\def\ps@pprintTitle{%
  \let\@oddhead\@empty
  \let\@evenhead\@empty
  \def\@oddfoot{\reset@font\hfil\thepage\hfil}
  \let\@evenfoot\@oddfoot
}
\makeatother

\usepackage{amssymb}
\usepackage{amsmath}
\newcommand{\N }{\ensuremath{\mathbb N} } 
\usepackage{tikz}
\usepackage{pgfplots}
\usetikzlibrary{shadows,shapes,matrix,arrows,petri}
\usepackage{wasysym}
\usepackage{xcolor}
\usepackage{graphicx}
\allowdisplaybreaks

\tikzset{weird fill/.style={append after command={
   \pgfextra
        \draw[sharp corners, fill=#1,draw=white]%
    (\tikzlastnode.west)%
    [rounded corners=8pt] |- (\tikzlastnode.north)%
    [rounded corners=1pt] -| (\tikzlastnode.east)%
    [rounded corners=5pt] |- (\tikzlastnode.south)%
    [rounded corners=0pt] -| (\tikzlastnode.west);
   \endpgfextra}}}

\newtheorem{theorem}{Theorem}
\newtheorem{definition}{Definition}
\newtheorem{lemma}[theorem]{Lemma}
\newtheorem{corollary}[theorem]{Corollary}
\newtheorem{proposition}[theorem]{Proposition}
\newdefinition{remark}{Remark}
\newproof{proof}{Proof}

\begin{document}

\begin{frontmatter}



\title{Asymptotically Optimal Threshold Bias for the
$(a : b)$ Maker-Breaker Minimum Degree, Connectivity and Hamiltonicity  Games} 


\author[label1]{Adnane Fouadi}
\author[label1,label2]{Mourad El Ouali}
\author[label2]{Anand Srivastav}
\affiliation[label1]{organization={Laboratory of Engineering Science, Faculty of Science, Ibn Zohr University},city={Agadir}, country={Morocco}}
            
\affiliation[label2]{organization={Department of Mathematics,  Christian-Albrechts-Universität zu Kiel},
 city={Kiel},country={Germany}}

\begin{abstract}
We study the $(a:b)$ Maker-Breaker subgraph game played on the edges of the complete graph $K_n$ on $n$ vertices, $n,a,b \in \N$ where the goal of Maker is to  build a copy of a specific fixed subgraph $H$, while Breaker aims to prevent Maker from building $H$.  Breaker starts. Maker wins, if he succeeds to build a copy of the subgraph under consideration, otherwise Breaker wins. In our work $H$ is a spanning graph with minimum degree $k=k(n)$, a connected spanning subgraph or a Hamiltonian subgraph. In the $(a:b)$ game in each round Maker chooses $a$ unclaimed edges of $K_n$ and Breaker chooses $b$ unclaimed edges. The number $b_H\in \N$ satisfying the condition that for any $b<b_H$ Maker has a winning strategy, while for any $b \geq b_H$ Breaker has a winning strategy is called the threshold bias of the $H$-subgraph game. The determination of $b_H$, if it exists, is the fundamental problem in such games. For the $k$-minimum-degree (1:b) game the threshold bias is between $(1-o(1))\frac{n}{\ln n}$ due to Chv{\'a}tal and Erd{\"o}s  (1978) and $(1+o(1))\frac{n}{\ln n}$ if $k<\frac{\ln \ln n}{3}$ due to Gebauer and Szab{\'o}  (2009). For the general $(a:b)$ game we show that for $a=o\left(\sqrt{\frac{n}{\ln(n)}}\right)$ the threshold bias is between $(1-o(1)) \frac{an}{a+\ln(n)}$ and $(1+o(1)) \frac{an}{a+\ln(n)}$, if $k=o(\ln(n))$ and it is between $(1-o(1))n$ and $n-k$, if $a=\Omega\left(\sqrt{\frac{n}{\ln(n)}}\right)$ and $k=o(a)$.
Hefetz,  Mikalacki and Stojakovic $(2012)$ showed that the  threshold bias for the (a:b) connectivity game is between $(1-o(1)) \frac{c n}{c+1}$
and $\min \left\{c n,(1+o(1)) \frac{2 n}{3}\right\}$ provided that 
$a=c\ln(n)$ and $c$ is a constant in $]0,1]$. For the case $c>1$, they showed that the threshold bias is between $(1-o(1)) \frac{c n}{c+1}$ and $(1+o(1)) \frac{2 c n}{2 c+1}$and posed as open problem to find a matching asymptotic term in the upper and lower bound. We resolve this problem, and show that for $(1+o(1)) \frac{c n}{c+1}\leq b $ Breaker wins for any $a=c\ln(n)$, $c>0$. Finally, we consider the Hamiltonicity $(a:b)$ game. For the $(1:b)$ game Krivelevich (2011) showed that   $(1+o(1))\frac{n}{\ln n}$ is the exact threshold bias. We are able to treat the general $(a:b)$ case: if  $a=o\left(\sqrt{\frac{n}{\ln(n)}}\right)$, then the asymptotic optimal bias is $\frac{a n}{a+\ln n}$, while for  $a=\Omega\left(\sqrt{\frac{n}{\ln(n)}}\right)$, the asymptotic optimal bias is $n$. Some new techniques are required to achieve our results. 
\end{abstract}



\begin{keyword}
Positional Games \sep  Maker-Breaker subgraph game \sep biased games \sep minimum-degree-$k$ game \sep Connectivity \sep Hamiltonicity \sep asymptotic optimal  bias


\end{keyword}

\end{frontmatter}




\section{Introduction}

Positional game theory studies two-player games of perfect information, usually played on discrete objects, ranging from popular leisure games like Tic-Tac-Toe and Hex to entirely abstract games played on graphs and hypergraphs. Over the last 40 years such games \cite{beck2008combinatorial,glazik2022new,hefetz2014positional, hancock2019maker, krivelevich2011critical, liebenau2022threshold, mikalavcki2017winning} has been  extensively studied.\\

 Let  $\mathcal{F}$  be a family of subsets of a finite set $X$. In the biased  $(a:b)$ Maker-Breaker game played on the hypergraph $(X,\mathcal{F})$ two players, called Maker and Breaker, alternately claim  $a$ and $b$
previously unclaimed elements  of $X$.
The set $X$ is called the board; the elements of $\mathcal{F} \subset 2^X$ are the winning sets. Maker wins the game, if he  is able to occupy a winning set at the end of the game, otherwise Breaker wins. \\

The  threshold bias for the $(a:b)$ game is defined as follows: given a  Maker-Breaker game $(X,\mathcal{F})$ and $a \geq 1$, let $b_0(a)$ be the unique
positive integer such that for $b < b_0(a)$  Maker has a winning strategy   but if $b \geq b_0(a)$ then there is a winning strategy for Breaker.  Note that for the Maker-Breaker subgraph game it is not known, and   in fact is still a challenging open question since the pioneering work of Chv{\'a}tal and Erd{\"o}s \cite{chvatal1978biased} in 1978 , whether or not the threshold  bias exists for subgraphs containing a cycle.  A breakthrough towards  the optimal threshold bias for the Maker-Breaker triangle game is certainly the recent paper of Glazik and Srivastav  \cite{glazik2022new}.\\

In this paper we  study  the following  variants of the Maker-Breaker game: the minimum-degree-$k$  game, the  connectivity game, and the  Hamiltonicity game. All three games are   played on the complete graph $K_n$ on $n$ vertices, where the board of these games is the set of the  edges of $K_n$. The winning sets  for the Maker are  all spanning subgraphs with minimum degree $k$,  and all connected  spanning subgraph, and   all Hamilton cycles, respectively.

\subsection{Previous work} 
The biased (1:b) connectivity game was introduced   by Chv{\'a}tal and Erd{\"o}s \cite{chvatal1978biased}.  Maker wins the connectivity game, if he is able to occupy the edges of a spanning tree. The authors proved that if
$b\geq (1+o(1))\,\, \frac{n}{\ln(n)}$, then Breaker wins the game  and if $b\leq \left( \frac{1}{4}-o(1)\right) \frac{n}{\ln(n)}$, Maker is able  to occupy a spanning tree. Since then the problem of finding the exact leading constant for the
threshold bias of the connectivity game has been one of the  famous open problems in Maker-Breaker games. The first improvement came with an argument of Beck \cite{beck1982remarks}:  Maker wins the connectivity game for $b\leq \left(\ln 2 - o(1)\right) \frac{n}{\ln(n)}$. Later Gebauer and Szab\'o \cite{gebauer2009asymptotic} established the asymptotic optimal bias for this game showing that Maker has a winning strategy, if $b\leq (1-o(1)) \frac{n}{\ln n}$ . The same authors showed that
 Maker can  achieve  not only a spanning tree, but even an induced subgraph with  minimum degree $k=k(n)<\frac{\ln \ln n}{3}$ by analyzing the $(1 : b)$ minimum-degree-$k$ game, if $b \leq (\ln n-\ln \ln n-(2 k+3)) \frac{n}{\ln ^{2} n}$. The techniques used in their proofs were applied in several subsequent papers to Hamiltonicity (resp. connectivity and strong connectivity) of graphs, digraphs, and randomly perturbed graphs (see e.g. \cite{clemens2021maker,frieze2021maker,HefetzMS12,krivelevich2011critical}).
 Hefetz, Mikalacki and Stojakovi{\'c} \cite{HefetzMS12} considered   the general $(a:b)$ connectivity game.  For $a=c\ln(n)$ where $c$ is a constant, $0<c\leq 1$, they proved that the  threshold bias for the $(a:b)$ connectivity game is between $(1-o(1)) \frac{c n}{c+1}$
and $\min \left\{c n,(1+o(1)) \frac{2 n}{3}\right\}$, and  for $c>1$, they showed that the threshold bias is between $(1-o(1)) \frac{c n}{c+1}$ and $(1+o(1)) \frac{2 c n}{2 c+1}$. The lower bound and upper bound are not matching in the asymptotic term. 
We resolve  this problem and show that for $(1+o(1))\frac{cn}{c+1} \leq b$ Breaker has a winning strategy for any $a=c \ln n$ and $c>0$.\\
 
The Maker-Breaker Hamiltonicity game   has a rich  history. It began with the paper of Chv{\'a}tal and Erd{\"o}s \cite{chvatal1978biased}, who studied unbiased Hamiltonicity  and showed that Maker wins for every sufficiently large $n$. For the (1:b) Hamiltonicity game, Bollob{\'a}s and Papaioannou \cite{bollobas1982biased}  proved that Maker is able to win  if $b\leq \frac{c \ln(n)}{\ln(\ln(n))}$, for some constant $c > 0$. This result was improved by Beck \cite{beck1985random}  who showed that  Maker wins the game if  $b \leq \left(\frac{\ln 2}{27}-o(1)\right) \frac{n}{\ln n}$. Krivelevich and Szab{\'o} \cite{krivelevich2008biased} gave a winning strategy for Maker if $b \leq (\ln 2-o(1)) \frac{n}{\ln n}$. On the other hand,  Breaker wins, if he is able to isolate some vertex, which was already treated in the paper of  Chv{\'a}tal and Erd{\"o}s \cite{chvatal1978biased}, thus  Breaker has a winning strategy for $(1+o(1))\frac{n}{\ln n} \leq b$.   Determining the threshold bias in this game has remained a challenging question for numerous years. In 2011, Krivelevich \cite{krivelevich2011critical} resolved this problem and showed that $(1+o(1))\frac{n}{\ln n}$ is the exact threshold bias.\\

\subsection{Our results}

\subsubsection{The $(a:b)$ Maker-Breaker minimum-degree-k game}

We show that  for
  $a=o\left( \sqrt{\frac{n}{\ln(n)}}\right)$ and $k=o(\ln(n))$ Maker wins, if $b\leq\left(1-o(1) \right)\frac{an}{a+\ln(n)}$  involving the parameter $\textit{a}$ into the threshold bias, and for $a=\Omega\left( \sqrt{\frac{n}{\ln(n)}}\right)$ and $k=o(a)$ he wins for $b\leq \left(1-o(1)\right) n$. Both generalize the result of Gebauer and Szab{\'o} \cite{gebauer2009asymptotic} to a wide range of \textit{a}.\\
  
Our Maker strategy, briefly called  min-deg-strategy, aims to build stars around each vertex.  For the analysis some new arguments are needed. Gebauer and Szab{\'o}  \cite{gebauer2009asymptotic} define the average danger of a multiset of vertices for Maker resp. 

In their approach a vertex is dangerous only once a time during the game. In our approach  we reflect that a vertex can become dangerous several times during the course of the game, thus  we calculate the average in accordance with the number of different elements and their eventual repeats. In this way the dynamics of the game w.r.t the chosen strategy of Maker is captured more specifically. Further  Gebauer and Szab{\'o}'s \cite{gebauer2009asymptotic} technique of calculating the average value of a multiset  leads to complications in controlling the change of values between  the average of two multisets. This article  adopts a different approach  that allows us to obtain the generalized form of Theorem 1.2 in \cite{gebauer2009asymptotic}  without the need to discuss the case where two multisets are equal or not.\\

Additionally, we bound  Breaker's vertex degrees by $(1-\delta)n$ for some constant $\delta>0$ as long as their Maker's degree is less than $k$. This result is crucial for the analysis of the $(a:b)$ Hamiltonicity game in section \ref{H}.\\

At this point the question is whether there is a Breaker strategy so that he wins if $b \geq \left(1+o(1) \right) \frac{an}{a+\ln(n)} $ in case of $a=o\left( \sqrt{\frac{n}{\ln n}} \right)$ and $k=o(\ln n)$ resp. if $b\geq\left(1+o(1) \right)n $ in case of $a=\Omega\left(\sqrt{\frac{n}{\ln n} }\right)$ and $k=o(a)$, which would establishing the asymptotically optimal threshold bias. We are able to prove these statements. Furthermore, for the $(a:b)$ connectivity game, our result for $k=1$, in the case $a \leq \ln(n)$ imply the Breaker's win, if $b \geq \left(1+o(1) \right) \frac{an}{a+\ln(n)} $, while in \cite{HefetzMS12} the weaker bound of 
 $b \geq \left(1+o(1) \right) \frac{an}{\ln(an)} $ is stated. 

Our proof relies on the use of the BoxMaker game technique of Chv{\'a}tal and Erd{\"o}s  \cite{chvatal1978biased}, where Breaker plays as BoxMaker. Gebauer and Szabo \cite{gebauer2009asymptotic} used this technique to give a wining strategy for Breaker the minimum degree (1:b) Maker-Breaker game, and Hefetz et al \cite{HefetzMS12} adapted the same technique for Breaker's win for the $(a:b)$ Connectivity Maker-Breaker Game.
The instance of the BoxMaker game is constructed first  as  a large clique in $\mathbb{K}_{n}$. The size of this clique is crucial for the determination of the threshold bias, because the larger the clique is, the more difficult it becomes for Maker to prevent the isolation of vetices by Breaker.  In previious approach Hefetz et al \cite{HefetzMS12} for $a=c\ln{n},\, c>0$  the clique size is $h=\left \lceil \frac{a n}{(a+1) \ln \left(a n\right)}\right \rceil$. We are able to give an analysis for the larger clique size, namely $h=\left \lceil \frac{ n}{2 \left(a+\ln (n)\right)}\right \rceil$. 

\subsubsection{The $(a:b)$ Maker-Breaker connectivity  game}

 Hefetz et al. \cite{HefetzMS12}  showed that if $a=c \ln n$ for some $0<c \leq 1$, the threshold bias $b_0(a)$ satisfies $(1-o(1)) \frac{c n}{c+1}<b_{0}(a)<\min \left\{c n,(1+o(1)) \frac{2 n}{3}\right\}$, and if $a=c \ln n$ for some $c>1$, then $(1-o(1)) \frac{c n}{c+1}<b_{0}(a)<(1+o(1)) \frac{2 c n}{2 c+1}$. They posed as an open problem the reduction of the gap between the lower bound and the upper bound on $b_0(a)$. 
Our result for the $(a:b)$-minimum degree game for $a=c\ln{n}, \, c>0 $ and $b\geq \left(1+o(1) \right) \frac{an}{\ln(an)} $ immediatly implies Breaker's win of the Connectivity game.

 \subsubsection{The $(a:b)$ Maker-Breaker Hamiltonicity game}

We prove the following results
\begin{itemize}
\item Let $a=o\left( \sqrt{\frac{n}{\ln(n)}}\right)$. Then,  Breaker wins the $(a:b)$ Hamiltonictian game if $b \geq \left(1+o(1)  \right) \frac{an}{a+\ln(n)}$ and Maker wins if  $b \leq \left(1-o(1)  \right) \frac{an}{a+\ln(n)}$.
\item Let $a=\Omega\left( \sqrt{\frac{n}{\ln(n)}}\right)$. Then,  Breaker wins the $(a:b)$ Hamiltonictian game if $b \geq \left(1+o(1)  \right) n$ and Maker wins if  $b \leq \left(1-o(1)  \right) n$.
\end{itemize}

In this kind of results  Breaker follows the simple strategy, namely isolating a particular vertex, that means no Maker-edge is incident in this vertex. Then obviously he wins the Hamiltonicity game. By using our result on the  minimum-degree-$k$  game ensuring boundness of Breaker's and Maker's degrees and the method of Krivelevich \cite{krivelevich2011critical} we prove that Breaker wins the game as stated above. In fact, we show that  Breaker succeed to construct an  expander by using a randomized strategy derived from the strategy adopted in Theorem \ref{main2}. We generalize the expander graph approach of Krivelevich \cite{krivelevich2011critical}  to the $(a:b)$ Hamiltonicity game, where we can show that due to $a>1$ the number of rounds for constructing an appropriate expander graph is at most $\frac{16 n}{a}$.

\textbf{Notation.} When not necessary we omit, ceiling and floor signs. For a positive integer $s$, $H_s=\sum_{j=1}^{s} \frac{1}{j}$ is the $s$th harmonic number.

\section{Minimum-Degree-k Game}
\subsection{Winning Strategy for Maker \label{MW}}
We define the danger of a vertex for Maker and Breaker, and a strategy which minimizes the danger for Maker to loose the game.
\begin{definition}

Let $G_M$,$G_B$ denote the subgraphs of $G$ with  edges taken by Maker, Breaker respectively.  $d_M(v)$ and $d_B(v)$ resp.
denote the degree of vertex $v$ in $G_M$ resp. $G_B$ for $v \in V$.   A vertex $v$ is called \textit{dangerous} if $d_M(v)\leq k-1$. Let $\mathcal{D}(v) = d_B(v) - \frac{2b}{a} d_M(v)$  be the danger of vertex $v$. A vertex $v \in V$ is called saturated , if $d_M(v)+d_B(v)=n-1$.

\end{definition}

\textbf{Min-Deg-Strategy:}  For $i\geq 1$, Maker's $i$th move consists of \textit{a} steps. For every $1 \leq j \leq a$, in the $j$th step of his $i$th move Maker chooses a   vertex $v_i^{(j)}$ whose danger is maximal and chooses an arbitrary edge incident
with $v_i^{(j)}$, not already taken. \\

\begin{remark}
In this strategy Maker is blocking vertices most dangerous for him, a natural approach.
\end{remark}

For the analysis of the Maker's Min-Deg-Strategy we introduce the concept of the average danger.

Let's consider the $s$th round of the game, $s\in \N$, $s\geq 2$, and let $v_s$ be a particular dangerous vertex.  For every $1 \leq i \leq s$, let $M_{i}$ and $B_{i}$ denote the $i$th move of Maker and of Breaker, respectively. By Maker's strategy the vertices $v_{1}^{(1)}, \ldots, v_{1}^{(a)}, v_{2}^{(1)}, \ldots, v_{2}^{(a)}, \ldots, v_{s-1}^{(1)}, \ldots, v_{s-1}^{(a)}$ were of maximum danger at the appropriate time, that is, just before Maker's $j$th step of his $i$th move. For every $1 \leq i \leq s-1$, let $A_{s-i}=\left\{v_{s-i}^{(1)}, \ldots, v_{s-i}^{(a)}, \ldots, v_{s-1}^{(1)}, \ldots, v_{s-1}^{(a)}, v_{s}\right\}$ denote the subset of vertices of $\left\{v_{1}^{(1)}, \ldots, v_{1}^{(a)}, v_{2}^{(1)}, \ldots, v_{2}^{(a)}, \ldots, v_{s-1}^{(1)}, \ldots, v_{s-1}^{(a)}, v_{s}\right\}$ that are dangerous just before Maker's $(s-i)$th move and let $A_{s}=\left\{v_{s}\right\}$. Note that the same vertex can occur several
times in the multiset $\left\{v_{1}^{(1)}, \ldots, v_{1}^{(a)}, v_{2}^{(1)}, \ldots, v_{2}^{(a)}, \ldots, v_{s-1}^{(1)}, \ldots, v_{s-1}^{(a)}, v_{s}\right\}$.

  \begin{definition}{(Average Danger)}
 We define the average danger value of sets as follows
\begin{enumerate}
\item[(i)] $\overline{\mathcal{D}}_{M}(A_{s-i})=\displaystyle \frac{\sum_{w \in A_{s-i}} \mathcal{D}(w) }{ai+1}$ is the average danger value of the vertices in the multiset $A_{s-i}$ for Maker,  before he starts his moves in round $s-i$.
 \item[(ii)] $\overline{\mathcal{D}}_{B}(A_{s-i})=\displaystyle \frac{\sum_{w' \in A_{s-i}} \mathcal{D}(w') }{ai+1}$ is the average danger value of the vertices in the multiset $A_{s-i}$ for Breaker, immediately before round $s-i$.
 \end{enumerate}
 \end{definition}
 
 Note that the round $s-i$ starts with Breaker  moves. Thus, $\overline{\mathcal{D}}_{M}(A_{s-i})$ and $\overline{\mathcal{D}}_{B}(A_{s-i})$, in general have different values.\\
 
  In the next two lemmas we state lower and upper bounds for the average danger. Maker plays according to the Min-Deg-Strategy. We illustrate the definition for rounds $s-i$ and $s-i+1$:

  \begin{center}
  \definecolor{ududff}{rgb}{0.8, 0.8, 1.0}
\definecolor{ferrarired}{rgb}{1.0,0.11,0.0}
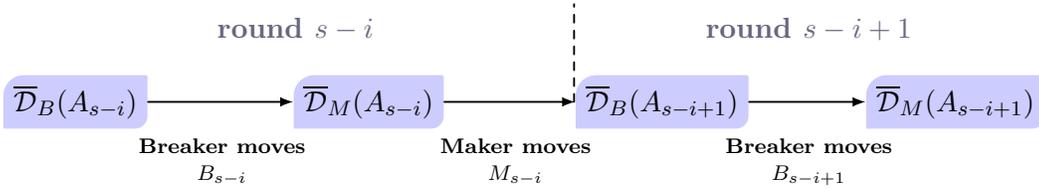
\begin{tikzpicture}[line cap=round,line join=round,>=latex,x=.65cm,y=.65cm]
\clip(-14,1) rectangle (8,5.2);

\node[weird fill=ududff](a) at (-12.5,3) {\small $\overline{\mathcal{D}}_B(A_{s-i})$};
\node[weird fill=ududff](b) at (-6.5,3) {\small $\overline{\mathcal{D}}_M(A_{s-i})$};
\node (a1) at (-9.5,2.1) {\scriptsize \textbf{Breaker moves}};
\node (a11) at (-9.5,1.5) {\scriptsize $B_{s-i}$};
\node[weird fill=ududff](c) at (-.5,3) {\small $\overline{\mathcal{D}}_B(A_{s-i+1})$};
\node (c1) at (-3.5,2.1) {\scriptsize \textbf{Maker moves}};
\node (c11) at (-3.5,1.5) {\scriptsize $M_{s-i}$};
\node[weird fill=ududff](d) at (5.5,3) {\small $\overline{\mathcal{D}}_M(A_{s-i+1})$};
\node (d1) at (2.5,2.1) {\scriptsize \textbf{Breaker moves}};
\node (d11) at (2.5,1.5) {\scriptsize $B_{s-i+1}$};

\draw[line width=.7pt,->] (a)--(b);
\draw[line width=.7pt,->] (b)--(c);
\draw[line width=.7pt,->] (c)--(d);
\draw[line width=.7pt,dashed] (-2.3,3)--(-2.3,5);
\node[color=ududff!50!black](c) at (-8,4.5) {\small \textbf{round $s-i$} };
\node[color=ududff!50!black](c) at (2.5,4.5) {\small \textbf{round $s-i+1$} };

\end{tikzpicture}

\end{center}

\begin{proposition}{}{ \label{0}}
For every $1 \leq i \leq s-1$, we have $\overline{\mathcal{D}}_{M}\left(A_{s-i}\right) \geq \overline{\mathcal{D}}_{B}\left(A_{s-i}\right) $ and
 $\overline{\mathcal{D}}_{M}\left(A_{s-i}\right) \geq \overline{\mathcal{D}}_{B}\left(A_{s-i+1}\right)$. 

\end{proposition}
\begin{proof}
The first inequality  is obvious, since during moves of Breaker the Maker degree of vertices in $A_{s-i}$ is unchanged, so the danger $\mathcal{D}(w)=d_B(w)-\frac{2b}{a}d_M(w)$, for $w \in A_{s-i}$ , can only increase. For the second inequality observe that during Maker's moves $M_{s-i}$ the danger $\mathcal{D}(w)$, $w \in A_{s-i}$, never increases\\

\end{proof}

\begin{definition}
We define a function  $g:\{1, \ldots, s\} \rightarrow \mathbb{N}$  where $g(i)$ is the number of edges Breaker has claimed during the first $i-1$ moves of the game  with both endpoints in $A_{i}$.
 \end{definition}
 \begin{lemma}\label{claim1}Let $p$ denote the number of edges $\{x, y\}$ claimed by Breaker during his move in round $s-i$ such that $x, y \in A_{s-i}$ with $1 \leq i \leq s-1$. Then, $p \leq a|A_{s-i+1}|+\left(\begin{array}{c}a \\ 2\end{array}\right)+g(s-i+1)-g(s-i)$.
\end{lemma}
\begin{proof}
 Before his $(s-i)$th move, Breaker has claimed exactly $g(s-i)$ edges with both their endpoints in $A_{s-i}$. Thus in round $s-i$, he has claimed exactly $g(s-i)+p$ edges with both their endpoints in $A_{s-i}$.  Note that $ A_{s-i}=\left\{v_{s-i}^{(1)}, \ldots, v_{s-i}^{(a)}\right\} \cup A_{s-i+1}.$ There can be at most $\left(\begin{array}{l}a \\ 2\end{array}\right)$ edges connecting any two vertices from $\left\{v_{s-i}^{(1)}, \ldots, v_{s-i}^{(a)}\right\}$. For each vertex in $\left\{v_{s-i}^{(1)}, \ldots, v_{s-i}^{(a)}\right\}$ there are  at most $|A_{s-i+1}|$ incident vertices in $A_{s-i+1}$. Thus  $g(s-i)+p \leq g(s-i+1)+a|A_{s-i+1}|+\left(\begin{array}{c}a \\ 2\end{array}\right)$, and $p \leq\left(\begin{array}{l}a \\ 2\end{array}\right)+a|A_{s-i+1}|+g(s-i+1)-g(s-i)$.
\end{proof}
 
 \begin{lemma}{}{\label{1}}
 The following two inequalities hold for every $1 \leq i \leq s-1$
 \begin{enumerate}
\item[(i)] $0 \leq \overline{\mathcal{D}}_{M}\left(A_{s-i}\right)-\overline{\mathcal{D}}_{B}\left(A_{s-i}\right) \leq \frac{2 b}{ai+1}  \leq \frac{2 b}{ai}  $
\item[(ii)]
$\overline{\mathcal{D}}_{M}\left(A_{s-i}\right)-\overline{\mathcal{D}}_{B}\left(A_{s-i}\right) \leq \frac{b-a^2+a+\left(\begin{array}{c}
a \\
2
\end{array}\right)+g(s-i+1)-g(s-i)}{a i}+a$
 \end{enumerate}

\end{lemma}
\begin{proof}
\begin{enumerate}
\item[(i)] The first inequality is clear by Proposition \ref{0}.  During  round $s-i$, Breaker claims $b$ edges. So, the sum of Breaker degrees of  vertices in  $A_{s-i}$ can increase in total by  $2 b$. Since  $\overline{\mathcal{D}}_{B}\left(A_{s-i}\right)$ is linear function in the danger values of vertices in $A_{s-i}$, $\overline{\mathcal{D}}_{B}\left(A_{s-i}\right)$ increases  at most by  an additive term of $\frac{2 b}{ai+1}$ during  Breaker's move.
\item[(ii)]  Let $q=b-p$, with $p$  defined as in  Lemma \ref{claim1}.   During Breaker's move in round $s-i$ the sum $\sum_{u \in A_{s-i}} d_{B}(u)- \frac{2b}{a} d_M(u)$  increases at most by $2 p+q=p+b$. Consequently,  $\overline{\mathcal{D}}_{B}\left(A_{s-i}\right)$ increases by at most an additive term of $\frac{b+p}{ai+1}$.\\ 
With Lemma \ref{claim1}  we can conclude the proof:
\begin{align*}
\overline{\mathcal{D}}_{M}\left(A_{s-i}\right)-\overline{\mathcal{D}}_{B}\left(A_{s-i}\right) &\leq \frac{b+p}{ai+1}\\
&\leq \frac{b+a|A_{s-i+1}|+\left(\begin{array}{c}
a \\
2
\end{array}\right)+g(s-i+1)-g(s-i)}{ai+1}\\& \leq \frac{b+a\left(a(i-1)+1\right)+\left(\begin{array}{c}
a \\
2
\end{array}\right)+g(s-i+1)-g(s-i)}{ai}\\
&\leq \frac{b-a^2+a+\left(\begin{array}{c}
a \\
2
\end{array}\right)+g(s-i+1)-g(s-i)}{a i}+a
\end{align*}
\end{enumerate}
\end{proof}

\begin{corollary}{\label{cor}}{}
For every $i$, $1\leq i \leq s-1$,  we have

 $\overline{\mathcal{D}}_{B}\left(A_{s-i}\right)\geq \overline{\mathcal{D}}_{B}\left(A_{s-i+1}\right)-\min\left\{\frac{2b}{ai},\frac{b-a^2+a+\left(\begin{array}{c}
a \\
2
\end{array}\right)+g(s-i+1)-g(s-i)}{ai}+a \right\} $

\end{corollary}
\begin{proof}
By Proposition \ref{0} we have $\overline{\mathcal{D}}_{B}\left(A_{s-i+1}\right)-\overline{\mathcal{D}}_{B}\left(A_{s-i}\right)\leq \overline{\mathcal{D}}_{M}\left(A_{s-i}\right)-\overline{\mathcal{D}}_{B}\left(A_{s-i}\right) $, and by Lemma \ref{1} we get the assertion of the corollary. 
\end{proof} 
\begin{remark}\label{Rem1}
We  use  certain well-known inequalities for the $k$th harmonic number $H_k=\sum_{i=1}^k1/i$, 
\begin{itemize}
\item[$\rm{a)}$] $\,\,\forall k\in \N ,\,\, \ln(k+1)\leq H_k\leq \ln(k)+1 $;
\item[$\rm{b)}$] $\,\,\forall i,j\in \N ,\,\, i\geq j\,\, \Longrightarrow\,\,  H_{j}-H_i \leq \ln(j)-\ln(i)$.
\end{itemize}
\end{remark}

\begin{lemma}{\label{propp}}
 $\overline{\mathcal{D}}_B(A_s)\leq\frac{2b}{a}\left(\ln(s-1)+1 \right). $
\end{lemma}
\begin{proof}
Before round one starts, no edges has been choosen neither by Breaker nor by Maker, so $\overline{\mathcal{D}}_{B}\left(A_{1}\right)=0$. We have
\begin{align}
0&=\overline{\mathcal{D}}_{B}\left(A_{1}\right)\nonumber\\
& \geq \overline{\mathcal{D}}_{B}\left(A_{2}\right)-\frac{2b}{a(s-1)} \tag{by Corollary \ref{cor} }\\& \,\, \vdots \nonumber \\&
 \geq \overline{\mathcal{D}}_{B}\left(A_{s}\right)-\frac{2b}{a(s-1)}-\cdots-\frac{2b}{a} \tag{by Corollary \ref{cor} } \\
&=\overline{\mathcal{D}}_B(A_s) -\frac{2b}{a} H_{s-1} \nonumber \\
&\geq \overline{\mathcal{D}}_B(A_s)-\frac{2b}{a}\left(\ln(s-1)+1 \right).
\end{align}
using $H_{s-1}\leq \ln(s-1)+1$. So   $\overline{\mathcal{D}}_B(A_s)\leq\frac{2b}{a}\left(\ln(s-1)+1 \right).  $
\end{proof}

\begin{lemma}{\label{prop}}
Let $r\in \N$.
\begin{enumerate}
\item[(i)] If $s> r$, then
$$\overline{\mathcal{D}}_B(A_s)\leq \frac{b}{a} \left(2\ln(s-1)-\ln(r)+1 \right) -\frac{(a-1)}{2} \ln\left(r \right)+ ra.$$
\item[(ii)] If $s \leq r$, then
$$\overline{\mathcal{D}}_B(A_s) \leq \frac{b}{a}\left(1+\ln(r) \right)-\frac{(a-1)}{2} \ln(r)+r a.$$
\end{enumerate}

\end{lemma}

\begin{proof}
\begin{enumerate}
\item[(i)] We have
\begin{align*}
0&=\overline{\mathcal{D}}_{B}\left(A_{1}\right)\\& \geq \overline{\mathcal{D}}_{B}\left(A_{2}\right)-\frac{2b}{a(s-1)}  \tag{by Corollary \ref{cor} }\\
&\,\,\vdots\\
&\geq \overline{\mathcal{D}}_{B}\left(A_{s-r}\right)-\frac{2b}{a(s-1)}-\cdots-\frac{2b}{a(r+1)} \qquad  \tag{ by Corollary \ref{cor} } \\
&\geq \overline{\mathcal{D}}_{B}\left(A_{s-r+1}\right)-\frac{2b}{a(s-1)}-\cdots-\frac{2b}{a(r+1)}\\&\,\,\,\,-\frac{b-a^2+a+\left(\begin{array}{c}
a \\
2
\end{array}\right)+g(s-r+1)-g(s-r)}{a r}-a  \tag{ by Corollary \ref{cor} }\\
&\,\,\vdots
\\
&\geq  \overline{\mathcal{D}}_{B}\left(A_{s}\right)-\frac{2b}{a(s-1)}-\cdots-\frac{2b}{a(r+1)}\\&\,\,\,\,-\frac{b-a^2+a+\left(\begin{array}{c}
a \\
2
\end{array}\right)+g(s-r+1)-g(s-r)}{a r}\\&\,\,\,\,-\cdots-\frac{b-a^2+a+\left(\begin{array}{c}
a \\
2
\end{array}\right)+g(s)-g(s-1)}{a }- r a   \tag{ by Corollary \ref{cor} }\\
&= \overline{\mathcal{D}}_B(A_s) -\frac{2b}{a(s-1)}-\cdots-\frac{2b}{a(r+1)}\\&\,\,\,\,-\frac{b-a^2+a+\left(\begin{array}{c}
a \\
2
\end{array}\right)+g(s-r+1)-g(s-r)}{a r}\\&\,\,\,\, -\cdots-\frac{b-a^2+a+\left(\begin{array}{c}
a \\
2
\end{array}\right)+g(s)-g(s-1)}{a }- r a\\ 
&= \overline{\mathcal{D}}_B(A_s) -\frac{2b}{a}\underbrace{\left(\frac{1}{s-1}+\cdots+\frac{1}{r+1}\right)}_{(*)} \\
&\qquad -\underbrace{\left(  \frac{b-a^2+a+\left(\begin{array}{c} a\\2 \end{array} \right)}{a r} +\cdots+\frac{b-a^2+a+\left(\begin{array}{c}a\\2 \end{array}\right)}{a}\right)}_{(**)}
\\
&\qquad -\underbrace{\left(\frac{g(s-r+1)-g(s-r)}{a r} +\cdots+\frac{g(s)-g(s-1)}{a }\right)}_{(***)}- r a\\
&= \overline{\mathcal{D}}_B(A_s)-\frac{2b}{a}\underbrace{\left(H_{s-1}-H_r \right)}_{=(*)}-\underbrace{\left(\frac{b}{a} H_r-aH_r+H_r+\frac{a-1}{2}H_r\right)}_{=(**)}\\&\qquad -\underbrace{\left(-\sum_{i=1}^{r-1}\frac{1}{ai(i+1)}g(s-i)+\frac{g(s)}{a}-\frac{g(s-r)}{ar}\right)}_{=(***)}-r a\\
&\geq \overline{\mathcal{D}}_B(A_s)-\frac{b}{a}\left(2\left(H_{s-1}-H_r \right)+ H_r \right)+\frac{a-1}{2}H_r-r a \tag{because of $(***) \geq 0$ due to  $g(s)=0$ , and  $g(s-i)\geq 0$ } \\
&\geq \overline{\mathcal{D}}_B(A_s)-\frac{b}{a}\left(2 (\ln(s-1)-\ln(r)) + \ln(r)+1 \right)+\frac{(a-1)}{2} \ln(r)-r a.   \\
\end{align*}
It follows that
$$\overline{\mathcal{D}}_B(A_s)\leq \frac{b}{a} \left(2 \ln(s-1) - \ln(r)+1 \right) -\frac{(a-1)}{2} \ln\left(r \right)+ ra.$$
\item[(ii)] We have
\begin{align*}
0&=\overline{\mathcal{D}}_{B}\left(A_{1}\right) \\& \geq \overline{\mathcal{D}}_{B}\left(A_{2}\right)-\frac{b-a^2+a+\left(\begin{array}{c}
a \\
2
\end{array}\right)+g(2)-g(1)}{a(s-1)}-a   \tag{by Corollary \ref{cor} } \\&\,\, \vdots \\
&\geq  \overline{\mathcal{D}}_{B}\left(A_{s}\right)-\frac{b-a^2+a+\left(\begin{array}{c}
a \\
2
\end{array}\right)+g(2)-g(1)}{a(s-1)}\\& \,\,\,\,-\frac{b-a^2+a+\left(\begin{array}{c}
a \\
2
\end{array}\right)+g(3)-g(2)}{a(s-2)}-\cdots \\& \,\,\,\,-\frac{b-a^2+a+\left(\begin{array}{c}
a \\
2
\end{array}\right)+g(s)-g(s-1)}{a}-(s-1)a \tag{by Corollary \ref{cor} } \\
&= \overline{\mathcal{D}}_B(A_s)-\frac{b}{a}H_{s-1}+aH_{s-1}-H_{s-1}-\frac{(a-1)}{2} H_{s-1}\\&\,\,\,\,+\sum_{i=1}^{s-2}\frac{1}{ai(i+1)}g(s-i)-\frac{g(s)}{a}+\frac{g(1)}{a(s-1)}-(s-1)a  \\ 
&\geq \overline{\mathcal{D}}_B(A_s)-\frac{b}{a}H_{s-1}+\frac{(a-1)}{2} H_{s-1}-(s-1) a  \tag{since  $g(s)=0$  and  $g(s-i)\geq 0$}\\ 
&> \overline{\mathcal{D}}_B(A_s)-\frac{b}{a}\left(1+\ln(s) \right)+\frac{(a-1)}{2} \ln(s)-s a \\
&\geq \overline{\mathcal{D}}_B(A_s)-\frac{b}{a}\left(1+\ln(r) \right)+\frac{(a-1)}{2} \ln(r)-r a.\\
\end{align*}
Since  $s \leq r$  and   $x\mapsto \frac{(a-1)}{2} \ln(x)-x a$   is decreasing on  $[1,+\infty[$, it follows that
$$ \overline{\mathcal{D}}_B(A_s) < \frac{b}{a}\left(1+\ln(r) \right)-\frac{(a-1)}{2} \ln(r)+r a$$
\end{enumerate}

\end{proof}

\begin{theorem}{\label{main}}
We assume that Maker plays according to the Min-Deg-Strategy.
\begin{itemize}
\item[(i)] For every $\varepsilon > 0$ and $k \in \N$ there is an $n_0(\varepsilon) \in \N$
such that for every $n \geq n_0(\varepsilon)$,  Maker can build a spanning subgraph of $K_n$ with minimum degree at least $ k=o(\ln(n))$ if $a = o\left(\sqrt{\frac{n}{\ln(n)}}\right)$ and $b \leq (1-\varepsilon)\frac{an}{a+\ln(n)}$.
\item[(ii)]  For every $\varepsilon > 0$ and $k \in \N$ there is an $n_0(\varepsilon) \in \N$
such that for every $n \geq n_0(\varepsilon)$, Maker can build a spanning subgraph of $K_n$ with minimum degree at least $k$ if $a = \Omega\left(\sqrt{\frac{n}{\ln(n)}}\right)$, $k=o(a)$ and   $b \leq  (1-\varepsilon)n$.
\end{itemize}
\end{theorem}
\begin{proof}
Assume for a moment that Maker loses the game, say in round $s$. Then, there is a vertex $v_s$ with $d_{M}(v_s)\leq k-1$. Before round $s$, Breaker has occupied at least $n-1-d_M(v_s)-b\geq n-k-b$ edges incident in $v_s$, so $d_B(v_s)\geq n-k-b$  at this time. Since  $A_s=\{v_s\}$, so  $\overline{\mathcal{D}}_B(A_s)=\mathcal{D}(v_s)$.  By definition of $\mathcal{D}(v_s)$, we get $\overline{\mathcal{D}}_B(A_s)=\mathcal{D}(v_s)=d_B(v_s)-\frac{2b}{a} d_M(v_s)\geq n-k-b-\frac{2 b (k-1)}{a}$.
 Maker's strategy calls for easing $a$ dangerous vertices in every round. Before round $s$ he has claimed  $a(s-1)$ edges. We show  $a(s-1)<kn$:  let $A=\{v_{1}^{(1)}, \ldots, v_{1}^{(a)}, v_{2}^{(1)}, \ldots, v_{2}^{(a)}, \ldots, v_{s-1}^{(1)}, \ldots, v_{s-1}^{(a)}\}$ be the multiset of dangerous vertices until round $s-1$. 
Then each vertex of $A$ is eased at most $k$ times, since $v_{s}$ is a dangerous vertex with degree less then $k-1$ in Maker's graph. Until round $s-1$ maker claimed exactly $a(s-1)$ edges. All these edges are incident in the set $A$ according to Maker's strategy. As said the Maker degree of a vertex in $A$ is at most $k-1$, and we get $a(s-1)< |A|k\leq nk$.

\begin{enumerate}
\item[(i)] Let $r=\frac{n}{a^2\ln n}$ and $b=(1-\varepsilon)\frac{an}{a+\ln(n)}$. By assumption, $a=o\left(\sqrt{\frac{n}{\ln(n)}} \right)$ and $k=o(\ln(n))$. We distinguish between two cases.\\

\textbf{Case 1:}  $s> r$. By Lemma \ref{prop} , we have 
$$ n-k-b-\frac{2 b (k-1)}{a}\leq D(v_s)   \leq \frac{b}{a} \left(2\ln(s-1)-\ln(r)+1 \right) -\frac{(a-1)}{2} \ln\left(r \right)+ ra.$$
With $s - 1 < \frac{kn}{a}$ and $r=\frac{n}{a^2\ln n}$, we get

\begin{align*}
 n-k-b-\frac{2 b (k-1)}{a}  &< \frac{b}{a}\left(2\ln(n)+2\ln(k)-2\ln(a)+\ln(\ln n)-\ln(n) + 2\ln(a)+1 \right)\\&\,\,\,\, -\frac{(a-1)}{2} \ln\left(\frac{n}{a^2\ln n} \right)+ \frac{n}{a\ln n}\\
&= \frac{b}{a}\left(\ln(n)+2\ln(k)+\ln(\ln n)+1 \right)\\&\,\,\,\, -\frac{(a-1)}{2} \ln\left(\frac{n}{a^2\ln n} \right)+ \frac{n}{a\ln n}.
\end{align*}  
leading to, 
\begin{equation}\label{eq1}
n-k+\frac{(a-1)}{2} \ln\left(\frac{n}{a^2\ln n} \right)- \frac{n}{a\ln n} < \frac{b}{a}\left(\ln(n)+2\ln(k)+\ln(\ln n) +a+2k-1\right) 
\end{equation} 
thus
$$b> a\frac{n-k+\frac{(a-1)}{2} \ln\left(\frac{n}{a^2 \ln(n)} \right)-  \frac{n}{a \ln(n)}}{\ln(n)+a+2k-1+2\ln(k)+\ln(\ln(n))}, $$
With  $b= (1-\varepsilon)\frac{an}{a+\ln(n)}$ , inequality \ref{eq1} is equivalent to

$$(1-\varepsilon)\frac{an}{a+\ln(n)}> \frac{an}{\ln(n)+a}\left(\frac{1-\frac{k}{n}+\frac{(a-1)}{2n} \ln\left(\frac{n}{a^2 \ln(n)} \right)-  \frac{1}{a \ln(n)}}{1+\frac{2k-1+2\ln(k)+\ln(\ln(n))}{\ln(n)+a}}\right), $$
which is a contradiction for $n$ sufficiently large.\\
\textbf{Case 2:} $s \leq r$. With Lemma \ref{prop} , we obtain 
$$ n-k-b-\frac{2 b (k-1)}{a}\leq D(v_s)    \leq   \frac{b}{a}\left(1+\ln(r) \right)-\frac{(a-1)}{2} \ln(r)+r a, $$
with  $r=\frac{n}{a^2\ln n}$, we get
\begin{align*} n-k-b-\frac{2 b (k-1)}{a}   &\leq   \frac{b}{a}\left(\ln(n)-2\ln(a)-\ln(\ln(n))+1 \right)\\&\,\,\,\,-\frac{(a-1)}{2} \ln\left(\frac{n}{a^2 \ln(n)} \right)+ \frac{n}{a \ln(n)}.
 \end{align*}
 Isolating $b$ gives 
$$b> a\frac{n-k+\frac{(a-1)}{2} \ln\left(\frac{n}{a^2 \ln(n)} \right)-  \frac{n}{a \ln(n)}}{\ln(n)+a+2k-1-2\ln(a)-\ln(\ln(n))}, $$
as the lower bound on $b$ tends to $\frac{an}{a+\ln(n)}$ for $n \rightarrow \infty$. Thus we get a contradiction to  $b= (1-\varepsilon)\frac{an}{a+\ln(n)}$  for $n$ sufficiently large.
\item[(ii)] Set $b=(1-\varepsilon)n$. We assume that $a=\Omega\left(\sqrt{\frac{n}{\ln(n)}} \right)$ and $k=o(a)$. With Lemma \ref{propp}, we have
$$n-k-b-\frac{2 b (k-1)}{a}\leq D(v_s)    \leq \frac{2b}{a}\left(\ln(s-1)+1 \right), $$
and with $s - 1 < \frac{kn}{a}$ we obtain
\begin{align*}
n-k-b-\frac{2 b (k-1)}{a}  < \frac{2b}{a}\left(\ln(k)+\ln(n)-\ln(a)+1\right).
\end{align*}
Thus
\begin{equation}\label{eq2}
b> \frac{a(n-k)}{2\left(\ln(n)+k+\frac{a}{2}+\ln(k)-\ln(a)\right)}. \end{equation}

With $b= (1-\varepsilon)n$, inequality (\ref{eq2}) gives
\begin{align*}
(1-\varepsilon)n &>  n\left(\frac{a-\frac{k}{n}}{2\left(\ln(n)+k+\frac{a}{2}+\ln(k)-\ln(a)\right)}\right)\\
=& n\left(\frac{a-\frac{k}{n}}{\left(2\ln(n)+2k+a+2\ln(k)-2\ln(a)\right)}\right)\\
=& n\left(\frac{1-\frac{k}{an}}{1+\left(\frac{2\ln(n)+2k+2\ln(k)-2\ln(a)}{a}\right)}\right),
\end{align*}
this is a contradiction for $n$ sufficiently large.
\end{enumerate}

\end{proof}

\begin{theorem}{\label{main2}}
Let  $0<\delta<\varepsilon<1$. If
\begin{itemize}
\item[(i)]   $a = o\left(\sqrt{\frac{n}{\ln(n)}}\right), k=o(\ln(n))$ and $b=(1-\varepsilon)\frac{an}{a+\ln(n)}$,
\item[(ii)]   $a = \Omega\left(\sqrt{\frac{n}{\ln(n)}}\right)$, $k=o(a)$ and   $b= (1-\varepsilon)n$,
\end{itemize}
 then according to the strategy of Theorem \ref{main}  Maker can  in at most $\frac{k n}{a}$ rounds ensure that Maker's graph has minimum degree $k$ such that Breaker-degree of $v$ is at most $(1-\delta) n$ as long as its Maker degree is less than $k$.
\end{theorem}
\begin{proof}
Maker plays according to the Min-Deg-Strategy.
Suppose that Breaker wins in round $s$, so  after
$B_{s}$ there is a vertex $v_s$ such that $d_B(v_s)>(1-\delta) n$ and $d_M(v_s)\leq k-1$. So $$\overline{D}_B(A_s)=\mathcal{D}(v_s)=d_B(v_s)-\frac{2b}{a}d_M(v_s)> (1-\delta) n-b-\frac{2(k-1)b}{a}.$$
Because this solely affects Maker's strategy, the conclusions of Lemmas \ref{propp} and \ref{prop} remain valid. In this situation, we can proceed in the same way as in the proof of Theorem \ref{main}, with the exception that the definition of $v_s$ has been changed  and we use $d_B(v_s)>(1-\delta) n$.

\begin{enumerate}
\item[(i)] Let $r=\frac{n}{a^2\ln n}$ and $b=(1-\varepsilon)\frac{an}{a+\ln(n)}$. By the assumption, $a=o\left(\sqrt{\frac{n}{\ln(n)}} \right)$ and $k=o(\ln(n))$. We distinguish between two cases.\\
\textbf{Case 1:}  $s> r$. In the same way as in the proof of Theorem \ref{main} , Case 1 in  (i), we  have
$$(1-\delta) n-b-\frac{2(k-1)b}{a} \leq D(v_s)   \leq \frac{b}{a} \left(2\ln(s-1)-\ln(r)+1 \right) -\frac{(a-1)}{2} \ln\left(r \right)+ ra.$$
With $s - 1 < \frac{kn}{a}$ and $r=\frac{n}{a^2\ln n}$, we get
\begin{align*}
(1-\delta) n-b-\frac{2(k-1)b}{a}  &<  \frac{b}{a}\left(\ln(n)+2\ln(k)+\ln(\ln n)+1 \right)\\&\,\,\,\, -\frac{(a-1)}{2} \ln\left(\frac{n}{a^2\ln n} \right)+ \frac{n}{a\ln n}.
\end{align*}  
It follows that
$$b> a\frac{(1-\delta) n+\frac{(a-1)}{2} \ln\left(\frac{n}{a^2 \ln(n)} \right)-  \frac{n}{a \ln(n)}}{\ln(n)+a+2k-1+2\ln(k)+\ln(\ln(n))}, $$

For $b=(1-\varepsilon)\frac{an}{a+\ln(n)}$ and since $(1-\varepsilon)<(1-\delta)$, the last inequality cannot hold for sufficiently large $n$.\\
\textbf{Case 2:} $s \leq r$. We obtain as above
\begin{align*} (1-\delta) n-b-\frac{2(k-1)b}{a}   &\leq   \frac{b}{a}\left(\ln(n)-2\ln(a)-\ln(\ln(n))+1 \right)\\&\,\,\,\,-\frac{(a-1)}{2} \ln\left(\frac{n}{a^2 \ln(n)} \right)+ \frac{n}{a \ln(n)}.
 \end{align*}
 Isolating $b$ gives 
$$b> a\frac{(1-\delta) n+\frac{(a-1)}{2} \ln\left(\frac{n}{a^2 \ln(n)} \right)-  \frac{n}{a \ln(n)}}{\ln(n)+a+2k-1-2\ln(a)-\ln(\ln(n))}, $$

for sufficiently large $n$ this leading to  contradiction for $b=(1-\varepsilon)\frac{an}{a+\ln(n)}$ and using $(1-\varepsilon)<(1-\delta)$.
\item[(ii)] Set $b=(1-\varepsilon)n$. We assume that $a=\Omega\left(\sqrt{\frac{n}{\ln(n)}} \right)$ and $k=o(a)$. In the same way as in the proof of Theorem \ref{main} in  (ii), We have
$$ (1-\delta) n-b-\frac{2(k-1)b}{a}  \leq D(v_s)    \leq \frac{2b}{a}\left(\ln(s-1)+1 \right), $$
and with $s - 1 < \frac{kn}{a}$ we obtain
\begin{align*}
(1-\delta) n-b-\frac{2(k-1)b}{a}   \leq \frac{2b}{a}\left(\ln(k)+\ln(n)-\ln(a)+1\right).
\end{align*}
It follows that
$$b> \frac{a(1-\delta) n}{2\left(\ln(n)+k+\frac{a}{2}+\ln(k)-\ln(a))\right)} $$
which, for large $n$, is a contradiction for $b=(1-\varepsilon)n$ using $(1-\varepsilon)<(1-\delta)$.
\end{enumerate}
\end{proof}

\subsection{A Winning Strategy for Breaker \label{BW}}

In this section we present a   strategy that allows  Breaker to win the minimum-degree-$k$ game.  For previous results we refer to the work of Hefetz, Mikalacki and Stojakovi{\'c}, \cite{HefetzMS12} and   Chv{\'a}tal and Erd{\"o}s  \cite{chvatal1978biased}, already discussed in the introduction.\\ 

\textbf{ The $(a : b)$ Box Game}\\

The $(a:b)$ Box game  is played on  $k$ pairwise-disjoint sets $ A_{1},\ldots ,A_{k}$ of different sizes, with $\sum_{i=1}^k |A_i|=t$. The sets are often called "boxes" and the elements are called balls. Two players called  BoxMaker and  BoxBreaker, claim  $a$ respectively $b$ elements from the sets.   BoxMaker wins if he can pick all balls in at least one box, while the objective of the  BoxBreaker is to prevent BoxMaker to succeed. Breaker starts.  This game can be described also with a hypergraph $\mathcal{H}=(V,\mathcal{E})$, where $V=\cup_{i=1}^{k} A_i$, and $\mathcal{E}=\left\{A_1,\ldots,A_k  \right\}$. Let $t:=|V|$.  $A_1, A_2,\ldots, A_k$ are pairwise distinct, and the sum of their sizes is $\sum^k_{i=1} |A_i| = t$.   The board of the
Box Game is said to be a canonical hypergraph of type $(k, t)$ if  $||A_i|-|A_j || \leq 1$ holds for every $1 \leq i, j \leq k$. This game is denoted by $B(k, t, a, b)$.\\

The Box Game was introduced in 1978 by Chv{\'a}tal and Erd{\"o}s \cite{chvatal1978biased} for one bias i.e., $a$ arbitrary, but $b=1$. For the $(a:b)$ Box Game Hamidoune and Las Vergnas \cite{hamidoune1987solution} provided  sufficient and necessary conditions so that BoxMaker wins the game where BoxBreaker  is the first player. In 2012 Hefetz et al  \cite{HefetzMS12} gave a sufficient condition for the win of BoxMaker where  the
first player is BoxMaker.\\ A BoxMaker winning strategy is as follows: When Breaker pics a ball in a specific box, Maker has lost this box. In the remaining boxes, the Maker claims balls to ensure a balancing condition (i.e $\left| |A_i|-|A_j| \right|\leq 1$). Maker continues in this manner until he reasches a box that has less  then $a$ balls.\\    

We will use a sufficient condition for BoxMaker’s win for the game $B (k , t , a, b )$ proved by Hefetz, Mikalacki and Stojakovi{\'c} \cite{HefetzMS12}. 

 Given positive integers a and b, let $f$ be the function defined as follows:

$$
f(k ; a, b):=\left\{\begin{array}{cl}
(k-1)(a+1) & , \text { if } 1 \leq k \leq b \\
k a & , \text { if } b<k \leq 2 b \\
\left\lfloor\frac{k(f(k-b ; a, b)+a-b)}{k-b}\right\rfloor, & \text { otherwise }
\end{array}\right.
$$

\begin{lemma}{(\textbf{Lemma 6 in \cite{HefetzMS12}}) \label{lem1} } Let $a, b$ and $k$ be positive integers satisfying $k>b$ and $a-b-1 \geq 0$. Then
$$
f(k ; a, b) \geq k a-1+\frac{k(a-b-1)}{b} \sum_{j=2}^{ \lceil k / b]-1} \frac{1}{j} .
$$
\end{lemma}

\begin{lemma}{}{}{(\textbf{Lemma 7 in \cite{HefetzMS12}}) \label{lem2} }   If  $t \leq f(k ; a, b)+a$,  then BoxMaker has a winning strategy for  $B(k, t, a, b)$. 
\end{lemma}
Here is our strategy for Breaker: 
Let $b=(1+\varepsilon) \frac{a n}{a+\ln ( n)}$,  $a=o(\sqrt{n})$ and $h:=\left \lceil \frac{ n}{2 \left(a+\ln (n)\right)}\right \rceil$ .  Breaker wins if and only if he is able to achieve $d_B(v)\geq n-k$  for some vertex $v$ in $K_{n}$.\\
He will operate in two steps. First, he claims all edges of a clique $C$ on at least $h$ vertices such that Maker cannot touch any vertex of $C$. In the second step,  he claims $n-k-h+1$ free edges incidents to some vertex $v$ in $C$. Thus, the degree of $v$ in the Maker's graph can't exceed $k$.  Let $|C|$ denote the number of vertices of $C$. \\
\textbf{Breaker's strategy:}\\
\textit{First step: Clique construction} \\
Set $C_1:=\emptyset$.\\
For $i=2,\cdots,h$ Breaker claims all edges of a clique $C_i$, as large as possible, such that the following conditions are satisfied.
 \begin{itemize}
 \item[$a)$]  $C_i$ is a clique in Breaker's graph such that $d_M(v)=0$ for every $v$ in $C_i$  just before Breaker's $i$th turn, so no Maker edge is incident in $v$. 
 \item[$b)$]  $|C_{i}|\geq |C_{i-1}|+1$.\\
 After clique $C_i$ has been built, Maker select $a$ edges.
 \end{itemize} 
 The construction stops if  a clique of cardinality at least $h$ has been built.\\
\textit{Second step: Playing the box game}\\ 
Let $C$ be the clique Breaker created in the first step, where $|V(C)| \geq h$ and $d_{M}(v)=0$ for every $v \in V(C)$. Just after constructing $C$, Maker can touch at most  $a$ vertices of $C$. Let $A$ be the set of these vertices.
\begin{itemize}
\item[c)] Set $V^*$ a subset of $C\backslash A$ of cardinality $h-a$.
\item[d)] For each node $v\in V^*$, let $E_{v}$ be the set of  edges incident in $v$ not taken by Maker and not belonging to $C$.
\item[e)] Define  the hypergraph $\mathcal{H}=(V,\mathcal{E})$ with $V=\displaystyle \bigcup_{v \in V^*} E_v$ and $\mathcal{E}:=\left\{E_v; v \in V^*\right\}$ as the instance of Box-Maker-Breaker game. So, the nodes of $\mathcal{H}$ are edges in $K_n$.  
\item[f)] Breaker plays first as a BoxMaker in the auxiliary box game $B(h-a, (h-a)(n-k-h+1), b, a)$.
\end{itemize}
\begin{lemma}\label{xx}
If Breaker playing as BoxMaker wins the box game $B(h-a, (h-a)(n-k-h+1), b, a)$, then he wins the minimum-degree-$k$ game.
\end{lemma}
\begin{proof}
Consider an arbitrary vertex $v\in V$ and $v$ corresponds to a Box owend by BoxMaker. Then at least $n-k-h-1$ edges are incident in $v$ and do not belong to the clique $C$. Since $|C|=h$, the number of Breaker edges incident in $v$ is at least $n-k-h-1-h+1=n-k$. Thus Maker has claimed at most $n-1 -n+k=k-1$ edges incident in $v$, and Breaker wins. 
\end{proof}

We will show that with this strategy Breaker wins the minimum-degree-$k$ game. The proof will also show that all steps a),b) and f) can be implemented correctly.
\begin{theorem}{\label{Breaker 1}}
For every $\varepsilon > 0$  there is a $n_0(\varepsilon) \in \N$
such that for every $n \geq n_0(\varepsilon)$, Breaker wins the $(a: b)$  minimum-degree-$k$ game  if   $k=o(n)$, $a=o(\sqrt{n})$ and $b \geq (1+\varepsilon) \frac{a n}{a+\ln ( n)}$.
\end{theorem}
\begin{proof}
The first step consists of $h$ moves \textit{at most} and  $C_1:= \emptyset$. Suppose Breaker has built the cliques $C_1,\ldots,C_{i}$, $i\in \{1,\ldots,h\}$. According to Maker's move at the end of the first step, he  has claimed   at most $a$ edges incident in $C$. \\\\
\textit{Claim 1.} There are at least  $a+1$ vertices $u_1,\ldots,u_{a+1}$ which are neither in any of the cliques $C_i$ nor in Maker's graph.\\\\
\textit{Proof of Claim 1.}   Breaker has created $C_i$,  Maker's graph contains at most $2a(h-1)$ vertices and $|V(C_i)|\leq h-1$. Since $a=o(\sqrt{n})$,  $n+2 \left(a+1 \right)\left(a+\ln(n)\right)=o(2n\ln(n))$, for sufficiently large $n$ we have  $2an+n+2 \left(a+1 \right)\left(a+\ln(n)\right)<2an+2n\ln(n)$, and this is equivalent to $\frac{(2a+1)n}{2\left(a+\ln(n)\right)}+a+1<n$. With our assumption on $h$, $h-1<\frac{n}{2\left(a+\ln(n)\right)}$, and we obtain $2a(h-1)+(h-1)+a+1<n$. So there are at least  $a+1$ vertices  $u_1,\ldots,u_{a+1}$ which neither belong to the clique $C_i$ nor to the  graph of Maker, and Claim 1 is proved.\\

By the assumption $a=o(\sqrt{n})$ we have  $ a^2+a+a\ln(n)+\ln(n)=o(n)$, thus $\frac{a(a+1)}{2}=o\left( \frac{an}{a+\ln(n)} \right)$. Further we have $(a+1)(h-1)\leq \frac{(a+1)n}{2\left(a+\ln(n)\right)}  \leq \frac{an}{a+\ln(n)}$. 
 Since  $b=(1+\varepsilon) \frac{a n}{a+\ln ( n)}$,  we conclude that $b> \left(\begin{array}{c}a+1 \\ 2\end{array} \right)+\left(a+1\right)(h-1)$ for sufficiently large $n$. Hence by Claim 1 Breaker, in his $(i+1)$th turn, is able to claim  all $\left(\begin{array}{c}a+1 \\ 2\end{array} \right)$ edges on the clique on $u_1,\ldots,u_{a+1}$. Further, he can claim  \textit{at most} $\left(a+1\right)\left|C_{i}\right|$ edges  between $\{u_1,\ldots,u_{a+1}\}$ and $C_i$. Let $X$ the nodes of $C_i$ and $C'_i$ the clique  constructed on $X\cup\{u_1,\ldots,u_{a+1}\}$. 

The  remaining $b-\left(\begin{array}{c}a+1 \\ 2\end{array} \right)-\left(a+1\right)\left|C_{i}\right|$ can be claimed  by Breaker arbitrarily.  On the other hand Maker can  claim at most $a$ vertices of $C'_{i}$ in his $i$th turn. Hence there is  at least one vertex   left, say $w$,  with Maker degree $0$ and it can be used to built new clique $C_{i+1}$  with  $\left|C_{i}\right|+1 \leq \left|C_{i+1}\right| $ and $d_M(v)=0$ for every vertex $v$ in $C_{i+1}$.

In the second step Breaker starts with a clique of size at least $h$ and $d_M(v)=0$ for $v\in V(C)$. Maker can  touch at most $a$ vertices of $C$. Let $A$ be a subset of $V(C)$ of cardinality $a$ that contains these nodes.  The aim of the  Breaker is to claim at least $n-k-h+1$  edges incident to some vertex $v_0 \in V(C)\backslash A$. 
 
Then,  $d_M(v_0)\leq (n-1)-(h-1)-(n-k-h+1) <k$, and Maker loses the game. We define $V^*=C\backslash A$. Clearly $|V^*|=h-a$. For each node $v\in V^*$, the  number of   edges  incident in $v$ not taken by  Maker and not belonging to $C$ is  $n-h$. Let $E_v$ be a subset of these edges of cardinality  $n-k-h+1$ for each node $v\in V^*$.   The Box-Maker-Breaker hypergraph is $\mathcal{H}=(V,\mathcal{E})$ with $V=\cup_{v \in V^*} E_v$ and $\mathcal{E}=\left\{E_v; v \in V^*\right\}$. Breaker plays first as a BoxMaker in the auxiliary Box Game $B(h-a, (h-a)(n-k-h+1), b, a)$.
 
By Lemma \ref{xx}, if BoxMaker has a winning strategy for $B(h-a, (h-a)(n-k-h+1), b, a)$, then Breaker wins the $(a: b)$ minimum-degree-$k$ game on $K_{n}$.

To prove that BoxMaker has a winning strategy for $B(h-a, (h-a)(n-k-h+1), b, a)$,  it is sufficient to show that $(h-a)(n-k-h+1) \leq f((h-a) ; b,  a)+b$,  according to Lemma \ref{lem2}. Because $a=o(\sqrt{n})$, we have  $h-a>a$ and $b-a-1 \geq 0$ for sufficiently large $n$. Hence with   Lemma \ref{lem1}  we get

$$
f(h-a ; b,  a) \geq (h-a) b-1+\frac{(h-a)(b- a-1)}{ a}  \sum_{i=2}^{ \lceil (h-a) / a \rceil-1} \frac{1}{i}.
$$

It is known that for $j\geq 1$, the $j$ th harmonic number $H_{j}=\sum_{i=1}^{j} 1 / i$ satisfies

$$\ln(j+1) <\ln j+1 / 2<H_{j}<\ln j+2 / 3 \qquad  \text{for sufficiently large $j$}. $$

 Therefore,

$$f((h-a) ; b,  a) \geq (h-a) b-1+\frac{(h-a)(b-a-1)}{a }\left(\ln \left \lceil \frac{h-a}{ a}\right \rceil-1\right).      $$

So, it remains to prove

$$
(h-a)(n-k-h+1) \leq (h-a) b-1+\frac{(h-a)(b-a-1)}{a}\left(\ln \left \lceil \frac{h-a}{ a}\right \rceil-1\right)+b,
$$

which is equivalent to

$$
a(n-k-h+1) \leq a b-\frac{a}{h-a}+(b-a-1)\left(\ln \left \lceil \frac{h-a}{ a}\right \rceil-1\right)+b \frac{a}{h-a},
$$

and the last inequality is equivalent to

$$
b \geq \frac{ a\left(n-k-h+\ln \left \lceil\frac{h-a}{ a}\right \rceil\right)+\ln \left \lceil\frac{h-a}{a}\right \rceil-1+\frac{ a}{h-a}}{ a+\ln \left \lceil\frac{h-a}{a}\right \rceil-1+\frac{ a}{h-a}} .
$$

As the lower bound on $b$ is equivalent to $\frac{an}{a+\ln(n)}$ and $b=(1+\varepsilon)\frac{an}{a+\ln(n)}$, the above inequality holds for sufficiently large $n$.

\end{proof}

We note a trivial bound on $b$ for  Breaker's win.

\begin{theorem}{\label{Breaker 2}}
If $b\geq n-k$, then Breaker wins the $(a: b)$  minimum-degree-$k$ game.
\end{theorem}
\begin{proof}
In his first turn, Breaker claims $n-k$ vertices which are incident to some vertex $v$. It follows that Maker can't touch the vertex $v$ at least $k$ times. Hence  Breaker wins.
\end{proof}

We determine the asymptotic optimal generalized threshold bias for this game, a straightforward conclusion from Theorem \ref{main} , Theorem  \ref{Breaker 1} and Theorem  \ref{Breaker 2}:
\begin{theorem}{\label{opt}}
\begin{itemize}
\item[(i)] If $a=o\left( \sqrt{\frac{n}{\ln(n)}}\right)$ and $k=o(\ln(n))$, then   the asymptotic optimal  threshold bias for this game is $\frac{an}{a+\ln(n)}$.
\item[(ii)] If $a=\Omega\left( \sqrt{\frac{n}{\ln(n)}}\right)$ and $k=o(a)$, then   the asymptotic optimal  threshold bias for this game is $n$.
\end{itemize}

\end{theorem}
Theorem \ref{Breaker 1} with $k=1$ immediately gives the following result for the Connectivity game

\begin{theorem}
\begin{description}
\item If $a=o\left( \sqrt{\frac{n}{\ln(n)}}\right)$, then   the asymptotic optimal generalized threshold bias for the Connectivity game game is $\frac{an}{a+\ln(n)}$.
\item If $a=\Omega\left( \sqrt{\frac{n}{\ln(n)}}\right)$, then  the asymptotic optimal generalized threshold bias for the Connectivity game game is $n$.
\end{description}
\end{theorem}
As discussed in the introduction, this resolves the open question of Hefetz et al \cite{HefetzMS12}.
\section{Hamiltonicity game \label{H} }

Due to Krivelevich \cite{krivelevich2011critical}, the critical bias for the (1:b) Hamiltonicity game on $K_n$ is asymptotically equal to   $\frac{n}{\ln(n)}$ which is the same for the (1:b) connectivity game. For  the  $(a:b)$  Hamiltonicity game we have:

\begin{theorem}{ \label{HB}}
\begin{itemize}
\item[(i)] Let $a=o\left( \sqrt{\frac{n}{\ln(n)}}\right)$. For every $\varepsilon > 0$  there is an $n_0(\varepsilon) \in \N$  such that for every $n \geq n_0(\varepsilon)$, Breaker wins the $(a:b)$ Hamiltonicity game if $b \geq \left(1+\varepsilon  \right) \frac{an}{a+\ln(n)}$ and Maker wins if $b \leq \left(1-\varepsilon \right) \frac{an}{a+\ln(n)}$.
\item[(ii)] Let $a=\Omega\left( \sqrt{\frac{n}{\ln(n)}}\right)$. For every $\varepsilon > 0$  there is an $n_0(\varepsilon) \in \N$ such that for every $n \geq n_0(\varepsilon)$, Breaker wins the $(a:b)$ Hamiltonicity game if $b \geq \left(1+\varepsilon   \right) n$ and Maker wins, if $b \leq \left(1-\varepsilon  \right) n$.
\end{itemize}

\end{theorem}
Note that Breaker's win as stated in Theorem \ref{HB} follows from Theorem \ref{Breaker 1} and Theorem \ref{Breaker 2} applied to $k=1$

We present a winning strategy for Maker. Its analysis relies on our Theorem \ref{main2} leading to the proof of Theorem \ref{HB}.

We need some results and  techniques from of Hefetz, Krivelevich,
Stojakovi{\'c} and Szabo \cite{hefetz2014positional}, who provided the winning strategy for Maker in the (1:b) Hamiltonicity game, for $b \leq \left(1-o(1)\right) \frac{ n}{\ln n}$.\\

To begin,  for a graph $G = (V,E)$ and a 
subset $U \subset V$ we denote by $N_G(U)$ the external neighborhood of $U$ in G, i.e. $N_G(U) =\left \{v \in V \backslash U : v \text{ has a neighbor in } U \right\}$.\\

\begin{definition}{(Definition 6.3.1. in \cite{hefetz2014positional})}
For a positive integer $k$, a graph $G=(V, E)$ is a $(k, 2)$-expander, or simply a $k$-expander if $\left|N_{G}(U)\right| \geq 2|U|$ for every subset $U \subset V$ of at most $k$ vertices.

\end{definition}

\begin{lemma}{(Lemma 6.3.2. in \cite{hefetz2014positional}) \label{comp}}
 Let $G=(V, E)$ be a $k$-expander. Then every connected component of $G$ has size at least $3k$.
\end{lemma}

\begin{definition}{(Definition 6.3.5. in \cite{hefetz2014positional})}
 Given a graph $G$, a non-edge $e=(u, v)$ of $G$ is called a booster if adding $e$ to $G$ creates a graph $G^{\prime}$, which is Hamiltonian or whose maximum path is longer than that of $G$.

\end{definition}

\begin{lemma}{(Corollary 6.3.6. in \cite{hefetz2014positional}) \label{boost}}
 Let $G$ be a connected non-Hamiltonian $k$-expander. Then $G$ has at least $\frac{(k+1)^{2}}{2}$ boosters.
\end{lemma}

 Let $a=o\left( \sqrt{\frac{n}{\ln(n)}}\right)$ and   $b= \left(1-\varepsilon\right) \frac{an}{a+\ln(n)}$ (respectively  $a=\Omega\left( \sqrt{\frac{n}{\ln(n)}}\right)$ and   $b=\left(1-\varepsilon \right) n$). We suppose that  $\delta \in \left]0,\min\left(\varepsilon,\frac{2^2}{e^7 10^4} \right) \right[$ and  $k=16$ in  Theorem \ref{main2}.\\ 
 
Let $k_0=\delta^5 n$. There are three stages in the Maker's strategy. Maker constructs a $k_0$-expander in at most $\frac{16 n}{a}$ moves in the first stage. Secondly, he ensures that his graph is connected in at most $O(n/ak_0)$ movements. Finally, he converts his graph to a Hamiltonian graph by making at most $\frac{n}{a}$ additional steps.\\

\textbf{Stage I – Creating an Expander.}\\
\textbf{Maker's Strategy:}  As long as there is a vertex of degree less than $16$ in Maker’s graph, Maker
chooses a vertex $v$ of degree less than $16$ in his graph with the largest danger value $\mathcal{D}(v)$ (breaking ties arbitrarily) and claims a random unclaimed edge $e$ incident to $v$ with respect to the uniform distribution.

\begin{lemma}{}
Maker has a random strategy to create a $k_0$-expander in at most $\frac{16 n}{a}$ moves, almost surely. 
\end{lemma} 

\begin{proof}

According to  Theorem \ref{main2}, the game's duration does not surpass $\frac{16 n}{a}$ and Breaker's degree of $v$ was at most $(1-\delta)n$, while Maker's degree of $v$ was at most $15$. \\

Assume for a moment  that Maker's graph is not a $k_0$-expander. So there is a subset $A$  of $V$ of size $|A| = i \leq k_0\,$ in Maker's graph  such that $N_M(A)$ is contained in a set $B$ of size $|B| = 2i-1$. We already stated that according to Theorem \ref{main2} the smallest degree in Maker's graph is $16$. We show that $i\geq 6$: Assume for a moment that $|A|\leq 5$. since the vertex-degree in $A$ in the Maker graph is greater than 16, every vertex in $A$ is adjacent to at least 16-4=12  vertices not belonging to A, thus $|B|\geq 12$. On the other hand $|B|=2i-1\leq 10-1=9$, which is a contradiction. 

Note that according to the strategy above Maker focus only on vertices with maximum danger. Since the Maker-degree of every vertex in $A$ is at least 16, by double counting we conclude that there are at least $8i$ Maker edges incident to $A$. We distinguish between two cases. \emph{First} at least $4i$ of these edges have one endpoint of maximum danger at the appropriate time in $A$ and the second endpoint  in $A \cup B$, \emph{second} at least $4i$ edges have one endpoint in $B$ and the second endpoint in $A$.  
Further we know that, according to Theorem \ref{main2}, Breaker's vertex-degree was at most $(1-\delta)n$, while Maker's vertex-degree was at most $15$. As a result, for every vertex in $V$ there are at least $\delta n-16$ unclaimed incident edges at that time. Regardless of the game's history, the probability that Maker chooses an edge at a vertex $v\in A$ with second endpoint belonging to $A\cup B$ --at that time-- is at most $\frac{|A\cup B|-1}{\delta n- 16}\leq \frac{3i-2}{\delta n- 16}$.  Hence the probability choosing at least $4i$ edges that are incidents to a vertex in $A$ with the other endpoint in $A\cup B$ is at most $\left(\frac{3i-2}{\delta n- 16}\right)^{4i}$. Now we consider choosing at least $4i$ edges that are incidents to $B$ with the other endpoint in $A$. Observe that there are $\left(\begin{array}{c}16|B| \\ 4 i\end{array}\right)$ possible choices for this situation. Regardless of the game's history, for $u\in B$,  the probability that its another endpoint is in $A$ is at most $\frac{|A|}{\delta n- 16}$. So, the probability that Maker chooses at least $4i$ edges that are incidents to $B$ with the another endpoint in $A$ is at most $\left(\begin{array}{c}16|B| \\ 4 i\end{array}\right)\left(\frac{i}{\delta n-15}\right)^{4 i}$.  Overall, given fixed $A$ and $B$ of sizes $|A|=i,|B|=2 i-1$ in Maker's graph, there exist at least $8 i$ edges incidentes to $A$, with probability at most
$$
\left(\frac{3 i-2}{\delta n-16}\right)^{4 i}+\left(\begin{array}{c}
32 i \\
4 i
\end{array}\right)\left(\frac{i}{\delta n-16}\right)^{4 i}<\left(\frac{10 e i}{\delta n}\right)^{4 i}
\qquad \qquad \left(\text{ since } \left(\begin{array}{l}n \\ k\end{array}\right) \leq\left(\frac{e n}{k}\right)^{k}\right) $$
So Maker fails to create a $k_0$-expander with probability at most
\begin{align*}
\sum_{i=6}^{k_{0}}\left(\begin{array}{l}
n \\
i
\end{array}\right)\left(\begin{array}{c}
n-i \\
2 i-1
\end{array}\right) \left(\left(\frac{3 i-2}{\delta n-16}\right)^{4 i}+\left(\begin{array}{c}
32 i \\
4 i
\end{array}\right)\left(\frac{i}{\delta n-16}\right)^{4 i}\right)&\leq \sum_{i=6}^{k_{0}}\left[\frac{e n}{i}\left(\frac{e n}{2 i}\right)^{2}\left(\frac{10 e i}{\delta n}\right)^{4}\right]^{i}\\&  \leq 
\sum_{i=6}^{k_{0}}\left[\frac{e^{7} 10^{4}}{2^{2}} \cdot \frac{1}{\delta^{4}} \cdot \frac{i}{n}\right]^{i}.
\end{align*}

 For $6 \leq i \leq \sqrt{n}$ we have $\left[\frac{e^{7} 10^{4}}{2^{2}} \cdot \frac{1}{\delta^{4}} \cdot \frac{i}{n}\right]^{i} \leq\left(O(1) n^{-1 / 2}\right)^6=O(n^{-3})$, while for $\sqrt{n} \leq i \leq k_0$ we have $\left[\frac{e^{7} 10^{4}}{2^{2}} \cdot \frac{1}{\delta^{4}} \cdot \frac{i}{n}\right]^{i} \leq\left(\frac{e^7 10^4}{2^2} \delta\right)^{\sqrt{n}}$. Let $\beta:= \frac{e^7 10^4}{2^2} \delta$, then $\beta <1$ since $\delta < \frac{2^2}{e^7 10^4}$. Hence,
$$\sum_{i=6}^{k_{0}}\left[\frac{e^{7} 10^{4}}{2^{2}} \cdot \frac{1}{\delta^{4}} \cdot \frac{i}{n}\right]^{i}\leq \sqrt{n}\,O(n^{-3})+ (k_{0}-\sqrt{n})\beta^{\sqrt{n}}=o\left(1\right).$$
As a result, Maker succeeds in building a $k_0$-expander in the first $\frac{16 n}{a}$ movements almost surely, and Lemma \ref{comp} is proved

\end{proof}
\begin{proof}{Theorem \ref{HB} (Maker's win)}\\
\textbf{Stage II – Creating a Connected Expander.}\\
According to Lemma \ref{comp}, every connected component in Maker's graph has a minimum size of $3 k_0$. With at least $9 k_{0}^{2}=\Theta\left(n^{2}\right)$ edges we can create a connected graph between any two such components. Maker's goal is to get an edge between each of the two components. The number of these components is at most $\frac{n}{3k_0}$. Hence, by the fact that the Maker plays $a$ moves in each round, he  requires at most $\frac{n}{3ak_0}$ moves to accomplish his objective. Therefore, Breaker has at most $\left(\frac{16 n}{a}+\frac{n}{3ak_0}\right) b<17  \frac{n^{2}}{\ln n}$ edges claimed if $b=(1-\varepsilon)\frac{an}{a+\ln(n)}$   (respectively $\left(\frac{16 n}{a}+\frac{n}{3ak_0}\right) b= O\left(n^{3/2} \sqrt{\ln(n)} \right)$  edges if $b= (1-\varepsilon)n$ ) on the board altogether. As a result, Breaker is powerless to stop Maker from attaining his purpose. \\
\textbf{Stage III – Completing a Hamiltonian Cycle.}\\
Maker builds a Hamiltonian cycle by gradually adding boosters. According to Lemma \ref{boost}, there are at least  $\frac{k_0^2}{2}$ non-edges that are boosters. Maker will add a maximum of $n$ boosters in at most $\frac{n}{a}$ moves. The duration of the three stages is at most $\frac{16 n}{a}+\frac{n}{3 a k_0}+\frac{n}{a} $ rounds during which Breaker selects at most
 $$\left(\frac{16 n}{a}+\frac{n}{3 a k_0}+\frac{n}{a}\right).b < \frac{18 n^2}{\ln n} \,\,\,\,\,\,\  \left( \text{ respectivlely} \left(\frac{16 n}{a}+\frac{n}{3 a k_0}+\frac{n}{a}\right).b= O\left(n^{3/2} \sqrt{\ln(n)} \right) \right)$$ 
  edges. Thus numbers, in both cases, are less than $k_0^2/2=\frac{1}{2} \delta^{10} n^2$ boosters of Maker. Thus Maker wins the game.
  
\end{proof}

\section{Conclusion}
We studied the $(a:b)$ Maker-Breaker minimum-degree-$k$ game and proved a lower and upper bound  leading to a direct generalization of the result presented by  Gebauer and Szab{\'o} \cite{gebauer2009asymptotic}.   Moreover, we answered  the open question stated by  Hefetz, Mikalacki, and Stojakovi{\'c}  \cite{HefetzMS12}  for the connectivity game. In addition, we generalized the bounds for the  (1:b) Hamiltonicity game to the $(a:b)$ Hamiltonicity game.
We hope that our result for Maker's winning strategy for the minimum-degree-$k$ game are helpful to analyze  $(a:b)$ Maker-Breaker games, for example the $(a:b)$ Connectivity and the $(a:b)$ Hamiltonicity game in oriented graphs.

\end{document}